\documentclass{amsart}[10pt]
\usepackage[utf8]{inputenc}
\usepackage[english]{babel}
\usepackage{eucal}
\usepackage{enumerate}
\usepackage{geometry}
\usepackage{amsthm}
\usepackage{amsfonts}
\usepackage{dsfont}
\usepackage{color}
\usepackage{todonotes}
\usepackage{amssymb}
\usepackage{amsmath}
\usepackage{graphicx}

\usepackage{xcolor}
\usepackage{graphicx}
\usepackage[T1]{fontenc}
\usepackage{ stmaryrd}
\usepackage{enumitem} 
\usepackage{rotating}
\usepackage{tikz}
\usepackage{dsfont}

\usepackage{hyperref}
\hypersetup{colorlinks=true,linkcolor=blue,citecolor=blue}    

\usepackage{ulem}
\newcommand{\IP}{{\mathbb P}}
\newcommand{\IE}{{\mathbb E}}

\newcommand{\cvlaw}{\stackrel{{ (d)}}{\longrightarrow}}
\newcommand{\cvIP}{\stackrel{{ \mathbb P}}{\longrightarrow}}
\newcommand{\eqlaw}{\stackrel{(d)}{=}}
\newcommand*\cvLdeux{\overset{L^2}{\longrightarrow}}

\newcommand{\dd}{\mathrm{d}}
\newcommand{\indep}{\perp \!\!\! \perp}

\newtheorem{theorem}{Theorem}[section]
\newtheorem{lemma}[theorem]{Lemma}
\newtheorem{proposition}[theorem]{Proposition}

\newtheorem*{theorem*}{Theorem}
\newtheorem*{lemma*}{Lemma}

\newtheorem*{proposition*}{Proposition}

\newtheorem{remark}[theorem]{Remark}

\theoremstyle{definition}

\newtheorem*{remark*}{Remark}
\newtheorem*{definition*}{Definition}

\theoremstyle{remark}

\begin{document}


  \title[On the Central Limit Theorem for 2d polymers]{On the central limit theorem for the log-partition function of 2D directed polymers}
  \author{Cl\'ement Cosco and Anna Donadini}
  \address{Cl\'{e}ment Cosco,
  Ceremade, Universite Paris Dauphine, Place du Mar\'{e}chal de Lattre de Tassigny, 75775 Paris Cedex 16, France}
  \email{clement.cosco@gmail.com}
  \address{Anna Donadini,
  Dipartimento di Matematica e Applicazioni, Università degli Studi di Milano-Bicocca, via
Cozzi 55, 20125 Milano, Italy}
  \email{a.donadini@campus.unimib.it}

\begin{abstract}
The log-partition function $ \log W_N(\beta)$ of the two-dimensional directed polymer in random environment is known to converge in distribution to a normal distribution when considering temperature in the subcritical regime $\beta=\beta_N=\hat{\beta}\sqrt{\pi/\log N}$, $\hat{\beta}\in (0,1)$ (\emph{Caravenna, Sun, Zygouras, Ann. Appl. Prob. (2017)}). In this paper, we present an elementary proof of this result relying on a  decoupling argument and the central limit theorem for sums of independent random variables. The argument is inspired by an analogy of the model to branching random walks.
\end{abstract}

\maketitle

\section{Introduction}

\subsection{The model and the result}
Directed polymers in random environment describe the behaviour of a long directed chain of monomers in presence of random impurities. The model was originally introduced in \cite{huse} and received its first mathematical treatment in \cite{imbrie}. In the most common setting, the trajectory of the polymer is given by a nearest-neighbour path $(S_n)_{n\in \mathbb N}$ on the $d$-dimensional lattice, while the impurities (also called the environment) are given by i.i.d.\ centered random variables $\omega(n,x)$, $n\in \mathbb{N}, x\in \mathbb{Z}^d$  of variance one satisfying $\lambda(\beta)=\log \IE[e^{\beta \omega(n,x)}]\in(0,\infty)$ for all $\beta >0$, where $(\IP,\IE)$ denotes the associated probability measure and expectation sign.
 The \emph{polymer measure} on the path space is then defined in the Gibbsian way by
\begin{equation}
\dd P_{N,\beta}^{\omega}(S)=\frac{e^{\beta \sum_{n=1}^N\omega(n,S_n)}}{Z_N(\beta)} \dd P(S),\label{eq:measure}
\end{equation}
where $\beta>0$ is the inverse temperature, $P=P_0$ with $(P_x,E_x)$ the probability measure and expectation of the simple random walk on $\mathbb Z^d$ started at point $x\in \mathbb Z^d$, and $Z_N(\beta)=E[e^{\beta \sum_{n=1}^N \omega(n,S_n)}]$ is the normalizing constant denoted as the \emph{partition function}.
A \emph{polymer path} is a realization of $(S_n){_{n\in\mathbb{N}}}$ under the polymer measure. Defined as such, the polymer is attracted to positive values of the environment $\omega(n,x)$ and repelled by negative values. 

In high dimension ($d\geq 3$), a phase transition occurs and there exists a critical value $\beta_c>0$ such that for $\beta\in (0,\beta_c)$, the polymer remains diffusive (weak disorder),  whereas for $\beta\in (\beta_c,\infty)$ the polymer localizes in small regions where the environment is particularly attractive (strong disorder). On the other hand, in lower dimension ($d=1,2$), the polymer always localizes i.e.\ $\beta_c = 0$.
The existence of a phase transition can be  characterized via the \emph{normalized partition function}:
\begin{equation}
  W_N(\beta)=E\left[e^{\sum_{n=1}^N (\beta \omega(n,S_n)-\lambda(\beta))}\right] = Z_N(\beta) e^{-N\lambda(\beta)},
\end{equation} 
which is a mean-one, positive martingale converging almost surely to a limit $W_{\infty}(\beta)$. Weak disorder then holds when $W_{\infty}(\beta)>0$ a.s,\ while strong disorder holds when $W_{\infty}(\beta)=0$. We refer to \cite{comets_book} for an extensive review of the model and further elaboration on these phenomena.

In the recent years, there has been much focus on studying the scaling properties of the model around the critical point $\beta_c=0$, when the dimension equals one or two. To this end, one can rescale the temperature as the volume grows by letting $\beta=\beta_N\to 0$. Non-degenerate limits of $W_N(\beta_N)$ can then be obtained by choosing
\begin{equation} \label{eq:def_betaN}
  \beta_N = \frac{\hat \beta}{\sqrt{R_N}}, \text{ where } R_N=E_0^{\otimes 2}\left[\sum_{n=1}^N\mathds{1}_{S_n^1=S_n^2}\right]\sim 
  \begin{cases} c \sqrt N & \text{ if } d=1,\\
  \frac{1}{\pi} \log N & \text{ if } d=2,
  \end{cases}
\end{equation}
where $S_N^1$ and $S_N^2$ are two independent copies of the simple random walk. In dimension $d=1$, it was observed in \cite{alberts} that for all $\hat \beta >0$, $W_N(\beta_N)\cvlaw \mathcal Z_{\hat \beta}$ where $\mathcal Z_{\hat \beta}$ is the solution to the stochastic heat equation at $x=0, t=1$ with flat initial condition, which by Hopf-Cole transformation relates to the celebrated KPZ equation. These equations belong to the so-called KPZ universality class, which gathers various models sharing non-standard scaling exponents and scaling limits (c.f.\ the surveys \cite{Corwin-noticesAMS,CorwinKPZ,Quastel2015}). As for example, the random variable $\mathcal Z_{\hat \beta}$ turns out to be a rather complex object whose distribution is related to the Tracy-Widom distribution \cite{ACQ11}. 

In dimension $d=2$, the phenomenology is quite different. In this case, Caravenna, Sun and Zygouras showed in the seminal paper \cite{caravenna1} that a phase transition (reminiscent of dimension $d\geq 3$) occurs, namely:
\begin{theorem}[\cite{caravenna1,caravenna3}] \label{th:main-th}
  Assume $d=2$ and set $\beta_N$ as in \eqref{eq:def_betaN}.
  For all $\hat \beta \geq 1$ one has $W_N(\beta_N) \cvIP 0$ as $N\to\infty$, while
  \begin{equation}
    \forall \hat \beta \in (0,1),\quad \log W_N(\beta_N) \cvlaw \mathcal{N}\left(-\frac{\lambda^2}{2},\lambda^2\right), \quad \lambda^2:=\log \left( \frac{1}{1-\hat\beta^2} \right).
  \label{eq:CLT0}
  \end{equation}
\end{theorem}
\begin{remark}
Equation \eqref{eq:CLT0} states that log-partition function (or free energy) satisfies a central limit theorem in the weak disorder regime $\hat \beta < 1$. It is in fact the main result of the theorem, as the convergence of $W_N(\beta_N)$ to $0$ for $\hat \beta \geq 1$ can easily be inferred from \eqref{eq:CLT0}, see e.g.\ the proof of \cite[Theorem 2.8]{caravenna1}. 
\end{remark}

The proof of \eqref{eq:CLT0} first proposed in \cite{caravenna1} relied on exhibiting a log-normal structure appearing in the chaos expansion of $W_N(\beta_N)$. 
A more direct proof was later given  in \cite{caravenna3} by approaching directly $\log W_N$ by a thoroughly defined chaos polynomial in order to apply a central limit theorem for chaos expansions.

The aim of the present paper is to provide a concise and elementary proof of \eqref{eq:CLT0}, relying only on simple tools. The idea is based on the observation that $W_N(\beta_N)$ factorizes asymptotically into  a product of independent smaller-scale partition functions, so that we can apply the usual central limit theorem for independent summands to $\log W_N(\beta_N)$. 

Our main ingredient is thus Theorem \ref{th:step1}, which states that the dependence of $W_N(\beta_N)$ to the environment decorrelates on different dyadic time windows. More precisely, one has $W_N(\beta_N)\approx \prod_{k\leq M} Z_k$, where the $Z_k$ are partition functions that live on separate dyadic time intervals $]t_{k-1},t_k]$ with $t_k=N^{k/M}, k\leq M$. 
This reduces the problem to showing that $\sum_{k\leq M} \log Z_k$ is close to a normal distribution. 
The key point is that the variables $\log Z_k$ are independent by definition, hence one can appeal to the central limit theorem for sums of independent random variables. To this end, we write $Z_k=1+U_k$  and use Taylor approximation to ensure that $\sum_{k\leq M} \log Z_k\approx \sum_{k\leq M} (U_k-\frac{1}{2}U_k^2)$ (Proposition \ref{prop:step2}). Then, we can get a good control on the $(2+\varepsilon)$ moments of $U_k$ (this is done via hypercontractivity, see Lemma \ref{lemma:fourth_moment}), so that the latter sum is easily shown to be close in distribution to $\mathcal N(-\lambda^2/2,\lambda^2)$  (Lemma \ref{lemma:wass} and Lemma \ref{lemma:trunc}).

Finally, we mention the work \cite{DG22} where ideas with a similar spirit have been used in the more general context of non-linear 2-dimensional stochastic heat equation.

\subsection{Connection to branching random walks and related works}
The idea behind the decoupling argument (Theorem \ref{th:step1}) originates from the analogy between two-dimensional polymers (in the regime $\hat \beta < 1$) and branching random walks, see \cite{CNZ25} where this analogy has been established at the level of extreme values  statistics and \cite{cosco} for a complementary discussion on this matter. We also refer to \cite{zeitouni} for an introduction to branching random walks and their relation to log-correlated fields. In light of Theorem \ref{th:step1}, the connection to branching random walks can be described as follows. If we define $W_N(x,\beta) = E_x[e^{\sum_{n=1}^N (\beta \omega(n,S_n)-\lambda(\beta))}]$ and $Z_k(x)= E_x[e^{\sum_{n=t_{k-1}+1}^{t_{k}} (\beta \omega(n,S_n)-\lambda(\beta))}]$,
then Theorem \ref{th:step1} suggests that the rescaled field
\begin{equation}\label{eq:Fieldh_N}
h_N(x) = \sqrt{\log N}\left(\log W_N\left(x\sqrt N,\beta_N\right)-\IE[\log W_N(\beta_N)]\right),
\end{equation}
can be approximated by 
$M_n(x)=\sum_{k=1}^M X_k(x)$, where $X_k(x) = \sqrt{\log N} (\log Z_k(x \sqrt N)-\IE[\log Z_k])$.
Heuristically, one can think of $M_n(x)$ as a branching random walk where the increments are given by the $X_k(x)$'s, which are independent variables in $k$ satisfying that $X_k(x)$ and $X_k(y)$ are close for small $|x-y|$ and decorrelate when $|x-y|$ is large. (in this analogy, the root of the tree would correspond to the level $k=M$ while leaves to the level $k=1$).   
We emphasize that the limiting variance of the increments $X_k(x)$ is strictly increasing in $k$ (see e.g.\ Theorem \ref{thm:1}),
hence, the asymptotic behavior of $h_n(x)$ should be  dictated by a model of branching random walk with decreasing variance, such as \cite{MZ12}. In fact, this heuristic has been confirmed in \cite{CNZ25} regarding extreme values of the field $h_n(x)$.
Interestingly, branching random walks with decreasing variance exhibit a different behavior, leading to different characteristic exponents, in comparison to their constant-variance counterpart, see \cite{CNZ25,MZ12} for precise statements.
 
We mention that in the regime $\hat \beta < 1$, one can prove the convergence of $h_N$ in \eqref{eq:Fieldh_N} as a field towards  a Gaussian log-correlated field. This has been shown under different strengths of generality, for continuous or discrete models, in a series of works \cite{ChDu18,caravenna2,Gu18KPZ2D,NaNa21}. It would be interesting to see if our techniques could as well lead to an elementary approach to this question.

The behavior at the critical point $\hat \beta = 1$ is quite different and of upmost interest. The quantity $\mathcal Z_N(x) = W_N(x\sqrt N)-1$ (which then converges pointwise to $-1$) was first shown by Bertini and Cancrini \cite{bertini_cancrini} to admit a non-trivial limiting variance after integration against test functions. Higher moments were later studied in \cite{CaSuZyCrit18,GQT,Chen24}, culminating with the recent breakthrough paper by Caravenna, Sun and Zygouras \cite{CaSuZyCrit21} that establishes the convergence in distribution of $\mathcal Z_N$. We refer to this paper for more information about this regime.

\subsection{Notation}
Throughout the proof, we always assume $\beta= \beta_N$ given in \eqref{eq:def_betaN}. In particular, we use the short-notations $W_N=W_N(\beta_N) = E_0\left[e^{ \sum_{n=1}^N (\beta_N \omega(n,S_n) - \lambda(\beta_N))}\right]$. We also use repeatedly that $S_n^2-S_n^1 \eqlaw S_{2n}$.

\subsection{Acknowledgments} The first author is deeply grateful to Francis Comets for suggesting to look at this question through one of his thoughtful stochastic calculus approaches. Although the present strategy differs slightly from his original plan, we hope that it would still have been to his taste. We are also very grateful to Ofer Zeitouni for discussions which lead to Theorem \ref{th:step1} and to Alberto Chiarini, Francesco Caravenna, Francesca Cottini and Shuta Nakajima for their careful reading and precious advice. 

\section{The decoupling argument}
In the following, we set for all $a,b\in[0,1]$ with $a\leq b$,
\begin{equation} \label{eq:defZab}
  Z_{a,b} = Z_{a,b,N} = E\left[e^{\sum_{n=\lceil N^a\rceil+1}^{\lceil N^b\rceil} (\beta_N \omega(n,S_n)-\lambda(\beta_N))}\right].
\end{equation}
(When $a=0$, we replace $\lceil N^a \rceil$ by $0$).

\begin{theorem}[Convergence of the second moment]\label{thm:1}
We have:
$\lim_{N\to\infty}\IE\left[Z_{a,b}^2\right]= e^{\lambda_{a,b}^2}$, 
where $\lambda_{a,b}^2:=\log \left( \frac{1- a\hat \beta^2}{1- b\hat \beta^2} \right)$. In particular, $\lim_{N\to\infty} \IE[W_N^2] = e^{\lambda^2}$.
\end{theorem}

\begin{proof}
The proof follows from the techniques of \cite[Sec.\ 3]{caravenna1}.  Let $\lambda_2(\beta):=\lambda (2\beta)-2\lambda(\beta)$. A simple computation (see e.g.\ \cite[p.36]{comets_book}) entails
\[
  \IE \left[ Z_{a,b}^2 \right] = E\left[e^{\sum_{n=\lceil N^a\rceil+1}^{\lceil N^b\rceil}\lambda_2(\beta_N)\mathds{1}_{S_{2n}=0}}\right].
\]
Let $p_n(x)=\IP_0(S_n=x)$. Since $e^{\lambda_2(\beta_N)\mathds{1}_{S_{2n}=0}}-1=\Lambda_N\mathds{1}_{S_{2n}=0}$ with $\Lambda_N=e^{\lambda_2(\beta_N)}-1$, we have
\begin{align}
\IE \left[ Z_{a,b}^2 \right] = E \left[ \prod_{n=\lceil N^a\rceil+1}^{\lceil N^b\rceil} \left(\Lambda_N\mathds{1}_{S_{2n}=0} +1\right) \right]&=1 + \sum_{k=1}^{\infty}\sum_{\lceil N^a\rceil+1\leq n_1 <\dots <n_k\leq \lceil N^b\rceil}\prod_{i=1}^{k}\Lambda_N p_{2(n_i-n_{i-1})}(0)\nonumber\\
& \leq 1 + \sum_{k=1}^{\infty}\Lambda_N^k (R_{\lceil N^b\rceil}-R_{\lceil N^a\rceil})(R_{\lceil N^b\rceil})^{k-1},\label{eq:boundOnEZ^2}
\end{align}
where $R_N = \sum_{n=1}^N p_{2n}(0)$ is defined in \eqref{eq:def_betaN}.
By the local limit theorem \cite[Sec. 1.2]{book},
\begin{equation} \label{eq:asymp_p2n}
p_{2n}(0) \sim_{N\to\infty} \frac{\pi}{n},\quad  R_{\lceil N^b \rceil}-R_{\lceil N^a \rceil} \sim_{N\to\infty} \frac{b-a}{\pi} \log N,\quad \Lambda_N\sim_{N\to\infty}\frac{\hat\beta^2\pi}{\log N},
\end{equation}
which implies that
\begin{align}
 \limsup_N \IE[Z_{a,b}^2] \leq 1+ \limsup_N \frac{\Lambda_N(R_{\lceil N^b\rceil}-R_{\lceil N^a\rceil})}{1-\Lambda_NR_{\lceil N^b\rceil}} = 1+\frac{(b-a)\hat\beta^2}{1-b\hat\beta^2} \label{eq:UBChaos} =e^{\lambda_{a,b}^2}.
\end{align}
The lower bound can be treated in a similar fashion, and we refer to \cite[Sec.\ 3]{caravenna1} for more details.

\end{proof}

Define the collection of (independent) random variables:
\begin{equation} \label{eq:defZ_k}
  Z_k = E\left[e^{ \sum_{n=t_{k-1}+1}^{t_{k}} (\beta_N\omega(n,S_n) - \lambda(\beta_N))}\right], \text{ with } t_0 = 0 \text{ and } t_k = \lceil N^{\frac{k}{M}}\rceil, 1\leq k\leq M.
  \end{equation}
(Note that $Z_k = Z_{\frac{k-1}{M},\frac{k}{M},N}$ with $Z_{a,b}$ as in \eqref{eq:defZab}.) 
\begin{theorem}[Dyadic time decoupling]
  \label{th:step1}
For all integers $M>1$ we have
\begin{equation}\label{eq:prop_step1}
W_N - \prod_{k=1}^{M} Z_k \xrightarrow[N \to \infty]{L^2} 0.
\end{equation}
\end{theorem}
\begin{remark}
We also refer to \cite[Lemma 6.2 and eq.\ (6.3)]{caravenna3} and \cite{CZ23} where similar decompositions have already occured.
\end{remark}

\begin{proof}
Recall that $t_1 = \lceil N^{1/M}\rceil$. Fix $M>1$ and consider for all $k\leq M$:
\begin{equation*}
    \tilde{Z}_k= E\left[e^{ \sum_{n=t_{k}+1}^{N} (\beta_N\omega(n,S_n) - \lambda(\beta_N))}\right].
\end{equation*}
Note in particular that $ Z_k \indep \tilde{Z}_k$.
We first prove that 
as $N\to\infty$,
\begin{equation}\label{eq:L2_Z0}
  W_N-Z_1\tilde{Z}_1\cvLdeux 0.
\end{equation}
Since $\IE[|W_N-Z_1\tilde{Z}_1|^2]=\IE[W_N^2]-2\IE[W_NZ_1\tilde{Z}_1]+\IE[Z_1^2\tilde{Z}_1^2]$ and by Theorem \ref{thm:1},
\begin{equation}
\IE[W_N^2]\to e^{\lambda^2}\label{eq:conv_WN_step1}\quad \text{and} \quad \IE[Z_1^2\tilde{Z}_1^2]=\IE[Z_1^2]\IE[\tilde{Z}_1^2]\to
e^{\lambda^2_{0,1/M}}e^{\lambda^2_{1/M,1}}=e^{\lambda^2}, 
\end{equation}
it suffices to prove that 
\begin{equation}\label{eq:star}
  \liminf_{N\to\infty}\IE[W_NZ_1\tilde{Z}_1]\geq  e^{\lambda^2}.
\end{equation}
Considering $S^1,S^2$ and $S^3$ three independent copies of $S$, we set $
  l_n^{1,j}:= \lambda_2(\beta_N)\mathds{1}_{S_n^1=S_n^j}$ and $l_{2n}:=\lambda_2(\beta_N)\mathds{1}_{S_{2n}=0}$. We see that:
\begin{equation*}
  \begin{split}
    \IE[W_NZ_1\tilde{Z}_1]&=E^{\otimes 3}\left[ e^{\sum_{n=1}^{t_1}l_n^{1,2}}e^{\sum_{n=t_1+1}^N l_n^{1,3}}\right]\\
    &= E^{\otimes 3}\left[e^{\sum_{n=1}^{t_1}l_n^{1,2}}E_{S_{t_1}^1,S_{t_1}^3}^{\otimes 2}\left[ e^{\sum_{n=1}^{N-t_1} l_n^{1,3}}\right]\right] \\
  &\geq E^{\otimes 3}\left[ e^{\sum_{n=1}^{t_1}l_n^{1,2}}\mathds{1}_{S_{t_1}^1-S_{t_1}^3\in A_N}E_{S_{t_1}^1-S_{t_1}^3}\left[ e^{\sum_{n=1}^{N-t_1}l_{2n}}\right]\right],
   \end{split}
  \end{equation*}
where we have used Markov’s property in the second line, and where we have set
\[
A_N=\{x\in\mathbb{Z}^2 : |x|\leq \alpha^{-1}\sqrt{t_1}, |x|=2k \text{ for some } k\in\mathbb{N}\},
\]
with $|\cdot|$  the $L^1$-norm and $\alpha>0$. Hence,
we obtain that
$\IE[W_N Z_1 \tilde{Z}_1]\geq \psi_N\varphi_N$,
where
\begin{gather*}
\psi_N=E^{\otimes 3}\left[ e^{\sum_{n=1}^{t_1}l_n^{1,2}}\mathds{1}_{S_{t_1}^1-S_{t_1}^3\in A_N}\right], \qquad \varphi_N=\inf_{x\in A_N}E_x\left[ e^{\sum_{n=1}^{N-t_1}l_{2n}}\right].
\end{gather*}
Since $e^{\lambda_{0,1/M}^2}e^{\lambda_{1/M,1}^2}=e^{\lambda^2}$, property \eqref{eq:star} would then follow from:
\begin{equation}\label{eq:liminf_psi}
(i)\,\liminf_{\alpha\to 0}\liminf_{N\to \infty}\psi_N \geq e^{\lambda_{0,1/M}^2},\quad (ii)\, \liminf_{\alpha\to 0}\liminf_{N\to\infty}\varphi_N\geq e^{\lambda_{1/M,1}^2}.
\end{equation}
We now focus on these two estimates. We begin with $\psi_N$. As $\mathds{1}_A=1-\mathds{1}_A^c$, we have
\begin{equation}\label{eq:psi}
\psi_N=E^{\otimes 2}\left[ e^{\sum_{n=1}^{t_1}l_n^{1,2}}\right]-E^{\otimes 3}\left[ e^{\sum_{n=1}^{t_1}l_n^{1,2}}\mathds{1}_{S_{t_1}^1-S_{t_1}^3\not\in A_N}\right]
\end{equation}
where by Theorem \ref{thm:1} 
$E^{\otimes 2}\left[ e^{\sum_{n=1}^{t_1}l_n^{1,2}}\right] \xrightarrow[N \to \infty]{} e^{\lambda_{0,1/M}^2}$.
We thus focus on the third term of \eqref{eq:psi}. By H\"older Inequality with $p,q\geq 1$ s.t. $\frac{1}{p}+\frac{1}{q}=1$, we have
\begin{equation*}
\begin{split}
E^{\otimes 3}\left[ e^{\sum_{n=1}^{t_1}l_n^{1,2}}\mathds{1}_{S_{t_1}^1-S_{t_1}^3\not\in A_N}\right]&\leq E^{\otimes 2}\left[\left(e^{\sum_{n=1}^{t_1}l_n^{1,2}}\right)^p\right]^{\frac{1}{p}}P^{\otimes 2}\left(S_{t_1}^1-S_{t_1}^3\not\in A_N\right)^{\frac{1}{q}} \\
&=E^{\otimes 2}\left[\left( e^{\sum_{n=1}^{t_1}l_n^{1,2}}\right)^p\right]^{\frac{1}{p}}P\left(\left|S_{2t_1}\right|>\alpha^{-1}N^{\frac{1}{2M}}\right)^{\frac{1}{q}},
\end{split}
\end{equation*}
so that, using Hoeffding's inequality which entails
$
P\left(\left|S_{2t_1}\right|>\alpha^{-1}N^{\frac{1}{2M}}\right)^{\frac{1}{q}} \leq e^{-c\alpha^{-2}/q},
$ we obtain
\begin{equation}\label{eq:conv_step1}
\limsup_{\alpha\to 0}\limsup_{N\to\infty}E^{\otimes 2}\left[ e^{\sum_{n=1}^{t_1}l_n^{1,2}}\mathds{1}_{S_{t_1}^1-S_{t_1}^3\not\in A_N}\right] =0,
\end{equation}
where we have used that for $p=p(\hat\beta)> 1$ small enough, we have (see Theorem \ref{thm:1})
\begin{equation*}
\limsup_{N\to\infty}E^{\otimes 2}\left[e^{p\lambda_2(\beta_N)\sum_{n=1}^{N}\mathds{1}_{S_n^1=S_n^2}}\right]^{\frac{1}{p}}<\infty.
\end{equation*} 
Hence, by \eqref{eq:psi} and \eqref{eq:conv_step1} we have proven \eqref{eq:liminf_psi}-(i).
Next, we prove \eqref{eq:liminf_psi}-(ii). For $x\in A_N$,
\begin{align}
&E_x\left[e^{\sum_{n=1}^{N-t_1}l_{2n}}\right]= 1 + \sum_{m=1}^{N-t_1}\Lambda_N^m\sum_{1\leq n_1 < \dots <n_m\leq N-t_1}E_x\left[\prod_{i=1}^m\mathds{1}_{S_{2n_i}=0}\right] \nonumber \\
&\geq 1+ \sum_{m=1}^{N-t_1}\Lambda_N^m\sum_{\alpha^{-3} t_1 \leq n_1 < \dots <n_m\leq N-t_1} p_{2n_1}(x)\prod_{i=2}^m p_{2(n_i-n_{i-1})}(0) \nonumber \\
&\overset{\text{LLT}}{\geq}1+ e^{-c\alpha}\left(\sum_{m=1}^{N-t_1}\Lambda_N^m\sum_{\alpha^{-3} t_1 \leq n_1 < \dots <n_m\leq N-t_1}\prod_{i=1}^m p_{2(n_i-n_{i-1})}(0)\right)(1+o_N(1))\label{eq:LLT_step1}\\
&\geq  e^{-c\alpha}E_0\left[e^{\sum_{\alpha^{-3}t_1}^{N-t_1}\lambda_2(\beta_N)\mathds{1}_{S_{2n}=0}}\right](1+o_N(1)), \label{eq:lastIneq}
\end{align}
where \eqref{eq:LLT_step1} holds by the local limit theorem \cite[Sec. 1.2]{book}: uniformly for all $x\in A_N$ and $n_1 \geq \alpha^{-3} t_1$,
\begin{equation*}
p_{2n_1}(x)\geq e^{-c\frac{|x|^2}{n_1}}p_{2n_1}(0) (1+o_N(1))\geq e^{-c\alpha}p_{2n_1}(0)(1+o_N(1)),
\end{equation*}
for some $c>0$. 
Moreover, for all fixed $\alpha >0$, using that for any arbitrary $\varepsilon>0$, we have $[N^{1/M + \varepsilon},N^{1-\varepsilon}]\subset [\alpha^{-3}t_1 , N-t_1]$ for $N$ large enough, we find that Theorem \ref{thm:1} yields
$\liminf_{N\to\infty} E_0[\exp\{\sum_{n=\alpha^{-3}t_1}^{N-t_1}l_{2n}\}] \geq e^{\lambda_{1/M,1}^2}$ .
Therefore, \eqref{eq:lastIneq} gives \eqref{eq:liminf_psi}-(ii).

Repeating the proof of \eqref{eq:L2_Z0}, one can show that $\forall k< M$ we have
$\tilde{Z}_{k} - Z_{k+1}\tilde{Z}_{k+1} \cvLdeux 0$.
Then, by independence of $\tilde{Z}_{k}-Z_{k+1}\tilde{Z}_{k+1}$ and $(Z_i)_{i\leq k}$ and since $\limsup_N\IE[Z_i^2]\leq e^{\lambda^2}$, one obtains
\begin{gather*}
W_N-\prod_{k=1}^{M}Z_k
=W_N-Z_1\tilde{Z}_1+Z_1(\tilde{Z}_1-Z_2\tilde{Z_2}) +\dots +\prod_{i=1}^{M-2}Z_i(\tilde{Z}_{M-1}-Z_{M-1}\tilde{Z}_{M-1})\xrightarrow[N \to \infty]{L^2} 0.
\end{gather*}
\end{proof}

\section{Estimates on the increments}

Define $U_k:=Z_k-1$. Notice that $U_k$ is, for every value of $k$, a centered random variable and that $U_1,\dots ,U_{M}$ are independent. The following lemmas give estimates on moments of the $U_k$'s.

\begin{lemma}[Variance of $U_k$]
  \label{lem:second_moment_Uk} There exists $C_{\hat \beta}>0$ such that
$
\limsup_{N\to\infty} \sup_{k\leq M}\IE\left[ U_k^2\right] \leq  \frac{C_{\hat \beta}}{M}$.
\end{lemma}
\begin{proof}
This follows from Theorem \ref{thm:1} as $\IE[U_k^2]=\IE[Z_k^2]-1$ and $\sup_{k\leq M} \{e^{\lambda^2_{(k-1)/M,k/M}}-1\}\leq \frac{C_{\hat \beta}}{M}$.
\end{proof}

In fact, by hypercontractivity of polynomial chaos (\cite[(3.10)]{caravenna2}), we can push this estimate to $2+\varepsilon_0$ moments of $U_k$, for some $\varepsilon_0$, as explained in the next lemma.

\begin{lemma}[Moment estimate]\label{lemma:fourth_moment}
There exist $\varepsilon_0=\varepsilon_0(\hat \beta)\in (0,1)$ and $c_{\hat \beta} <\infty$ such that $\forall M>0$:
\begin{equation}\label{eq:fourth_moment}
  \limsup_{N\to\infty} \sup_{k\leq M} \IE[\left|U_k\right|^{2+\varepsilon_0}]\leq \frac{c_{\hat \beta}}{M^{1+\frac{\varepsilon_0}{2}}}.
\end{equation}
\end{lemma}

\begin{proof}
Similarly to the line above \eqref{eq:boundOnEZ^2}, one can write $U_k = \sum_{m=0}^\infty X_m^{(N)}$, 
where $X_0^{(N)}=0$ and \[ X_m^{(N)}:=\sum_{t_k+1\leq n_1 <\dots <n_m\leq t_{k+1}}E\left[\prod_{i=1}^m\left( e^{\beta_N\omega(n_i,S_{n_i})-\lambda(\beta_N)}-1 \right)\right].\]
Now we notice that the {$X_m^{(N)}$'s are multilinear polynomials of degree $m$ in the centered i.i.d.\ variables $\xi_{n_i}^{(N)}:=\Lambda_N^{-1/2}\left( e^{\beta_N\omega(n_i,S_{n_i})-\lambda(\beta_N)}-1 \right)$ with unitary variance (see \cite[eq. (2.16)]{caravenna2} for an analogous definition)}. Thus, we can apply hypercontractivity for polynomial chaos (see \cite[eq. (3.10) and Appendix B]{caravenna2}) and deduce that for every $\varepsilon >0$ there exists a constant $c_\varepsilon$ uniform in $N$ such that $c_\varepsilon\to 1$ as $\varepsilon\to 0$ and
\begin{equation*}
\begin{split}
\IE\left[ |U_k|^{2+\varepsilon} \right] &= \IE \left[ \left| \sum_{m=1}^\infty X_m^{(N)} \right|^{2+\varepsilon}\right]\leq \left( \sum_{m=1}^\infty c_\varepsilon^{2m}\IE\left[ (X_m^{(N)})^2 \right] \right)^{1+\frac{\varepsilon}{2}}\\
&=\left(\sum_{m=1}^\infty c_\varepsilon^{2m} \sum_{t_k+1\leq n_1 <\dots <n_m \leq t_{k+1}}\IE\left[\prod_{i=1}^m\left( \Lambda_N \mathds{1}_{S_{2(n_i-n_{i-1})}=0} \right)\right] \right)^{1+\frac{\varepsilon}{2}}.
\end{split}
\end{equation*}
At this point we can choose $\varepsilon_0=\varepsilon_0(\hat\beta)\in (0,1)$ such that $c_{\varepsilon_0}\hat\beta <1$ and conclude, by \eqref{eq:boundOnEZ^2} that
\begin{equation*}
\IE\left[ |U_k|^{2+\varepsilon_0} \right] \leq \left( \frac{c_{\varepsilon_0}^2\Lambda_N(R_{t_{k+1}}-R_{t_k})}{1-c_{\varepsilon_0}^2\Lambda_N R_N} \right)^{1+\frac{\varepsilon_0}{2}}
\leq \frac{c_{\hat\beta}}{M^{1+\frac{\varepsilon_0}{2}}},
\end{equation*}
for $N$ large enough, where the second inequality is uniform in $k\leq M$ and is obtained using \eqref{eq:asymp_p2n}.
\end{proof}

From this point on, we fix $\varepsilon_0 \in (0,1)$ such that \eqref{eq:fourth_moment} holds. Recall that $Z_k=U_k+1$.
\begin{proposition}\label{prop:step2} For all $\delta >0$,
\begin{equation}\label{eq:prop_step2}
\limsup_{M\to\infty}\limsup_{N\to\infty}\IP\left(\left|\sum_{k=1}^{M}\log(1+U_k)-\sum_{k=1}^{M}\left(U_k-\frac{U_k^2}{2}\right)\right|\geq\delta\right)=0.
\end{equation}
Moreover,
\begin{equation} \label{eq:new}
\limsup_{M\to\infty}\limsup_{N\to\infty} \IP\left(\left|\log W_N - \sum_{k=1}^M \log(1+U_k)\right| \geq \delta \right) =0.
\end{equation}
\end{proposition}

\begin{proof}
Write $f(U):=\sum_{k=1}^{M}\log(1+U_k)-(U_k-\frac{1}{2}U_k^2)$ and set
$F_{M}=\bigcap_{k\leq M}\{ |U_k|<1/2 \}$. Then:
\begin{equation} \label{eq:bornefU>eps}
\IP(|f(U)|\geq \delta)=\IP(|f(U)|\geq \delta,F_{M})+\IP(|f(U)|\geq \delta,F_{M}^c).
\end{equation}
By Taylor approximation
$\left|\log(1+x)-\left(x-{x^2}/2\right)\right|\leq C |x|^{2+\varepsilon_0}$ for all $|x|\leq 1/2$, 
hence
\begin{gather}
\IP(|f(U)|\geq \delta,F_{M}){\leq}\IP\left(C\sum_{k=1}^{M}|U_k|^{2+\varepsilon_0}\geq \delta\right)
\overset{\text{Markov}}{\leq}C\delta^{-1}{\IE\left[\sum_{k=1}^{M}|U_k|^{2+\varepsilon_0}\right]}\overset{\eqref{eq:fourth_moment}}{\leq} \frac{Cc_{\hat\beta}}{\delta M^{\frac{\varepsilon_0}{2}}},\label{eq:term1}
\end{gather}
for $N$ large enough.
Furthermore, by the union bound,
\begin{equation} \label{eq:term2}
  \IP(F_{M}^c)\leq M\sup_{k\leq M}\IP(|U_k|\geq 1/2)
  \leq M2^{2+\varepsilon_0}\sup_{k\leq M}\IE\left[|U_k|^{2+\varepsilon_0}\right] \overset{\eqref{eq:fourth_moment}}{\leq} \frac{2^{2+\varepsilon_0}c_{\hat\beta}}{M^{\frac{\varepsilon_0}{2}}}.
\end{equation}
With \eqref{eq:bornefU>eps} and \eqref{eq:term1}, this entails \eqref{eq:prop_step2}. 

We turn to \eqref{eq:new}. Let $g(U) =\log W_N - \log Z$ where $Z=\prod_{k=1}^M Z_k$. We use that:
\begin{equation} \label{eq:FMWNZ}
  \IP(|g(U)|\geq \delta,F_M) \leq  \IP(|g(U)|\geq \delta,|W_N-Z| < 2^{-(M+1)}, F_M)+\IP(|W_N-Z| > 2^{-(M+1)}).
\end{equation}
Now observe that on $F_M$, we have $Z \geq 2^{-M}$ (recall that $Z_k=U_k+1$). Therefore, given that $\log$ is $\gamma_M$-Lipschitz continuous on $[2^{-(M+1)},\infty)$ for some $\gamma_M>0$, by the definition of $g(U)$ we have
\begin{equation} \label{eq:lastEstimate}
  \IP(|g(U)|\geq \delta,|W_N-Z| < 2^{-(M+1)},F_M)\leq \IP(|W_N -Z|\geq \delta\gamma_M^{-1}).
\end{equation}
Since Theorem \ref{th:step1} entails that $\limsup_N \IP(|W_N-Z|
\geq \varepsilon) = 0$ for any $\varepsilon >0$, we obtain \eqref{eq:new} by  \eqref{eq:FMWNZ}, \eqref{eq:lastEstimate} and \eqref{eq:term2}.
\end{proof}

\section{Use of the central limit theorem and law of large numbers}

In this section, we argue that $\sum_{k=1}^{M}\left( U_k-\frac{U_k^2}{2} \right)$ is well approximated in distribution by $\mathcal{N}(-\frac{\lambda^2}{2},\lambda^2)$. Below, $\mathcal W_1$ denotes the $L^1$ 
-Wasserstein distance between two measures (see Appendix \ref{appendix}).

A first important step is to establish convergence of the variance to $\lambda^2$.

\begin{lemma}\label{lemma:variance}
Let $\lambda_{M,N}^2:=\IE[\sum_{k=1}^M U_k^2]$. Then we have that
\begin{equation}\label{eq:variance}
\limsup_{M\to \infty}\limsup_{N\to\infty}|\lambda_{M,N}^2-\lambda^2| = 0.
\end{equation}
\end{lemma}

\begin{proof}
\eqref{eq:variance} follows from Theorem \ref{thm:1} and Riemann summation, which entail
\begin{align}
\IE\left[ \sum_{k=1}^MU_k^2\right] & = \sum_{k=1}^M\left( \IE\left[ Z_k^2\right]-1\right) \xrightarrow[N \to \infty]{}
\sum_{k=1}^M \frac{1}{M}\left( \frac{\hat\beta^2}{1- \hat\beta^2\frac{k+1}{M}} \right)\xrightarrow[M \to \infty]{} \int_0^1\frac{\hat\beta^2}{1-\hat\beta^2x}dx\nonumber=\lambda^2.\label{eq:second_moment}
\end{align}
\end{proof}

\begin{lemma}\label{lemma:wass}
We have that for all $M>0$,
\begin{equation}\label{eq:wass}
\limsup_{N\to\infty}\mathcal{W}_1\left(\mathcal{L}\left(\sum_{k=1}^{M}U_k\right),\mathcal{N}(0,\lambda^2)\right)\leq \frac{c_{\hat\beta}'}{\lambda_{M,N}^{\frac{1+\varepsilon_0}{2}}\,M^{\frac{\varepsilon_0}{2}}} + |\lambda_{M,N}-\lambda \,|^2.
\end{equation}
\end{lemma}

\begin{proof}
First, we notice that
\begin{equation*}
\mathcal{W}_1\left(\mathcal{L}\left(\sum_{k=1}^{M}U_k\right),\mathcal{N}(0,\lambda^2)\right)\leq \mathcal{W}_1\left(\mathcal{L}\left(\sum_{k=1}^MU_k\right),\mathcal{N}(0,\lambda_{M,N}^2)\right) + |\lambda_{M,N}-\lambda \,|^2,
\end{equation*}
where we used that $\mathcal{W}_1$ is bounded by $\mathcal{W}_2$ and that
\begin{equation*}
\mathcal{W}_2\left(\mathcal{N}(0,\lambda_{M,N}^2),\mathcal{N}(0,\lambda^2)\right)  = |\lambda_{M,N}-\lambda \,|^2.
\end{equation*}
Then, by the estimate on the Wasserstein distance between the law of the sum of independent random variables and the law of a Gaussian stated in \cite[(1.3b)]{rio}, we have that 
\begin{equation*}
\mathcal{W}_1\left(\mathcal{L}\left(\sum_{k=1}^MU_k\right),\mathcal{N}(0,\lambda_{M,N}^2)\right)\leq \frac{C}{\lambda_{M,N}^{\frac{1+\varepsilon_0}{2}}}\sum_{k=1}^M \IE[U_k^{2+\varepsilon_0}] \overset{\eqref{eq:fourth_moment}}{\leq} \frac{c_{\hat\beta}'}{\lambda_{M,N}^{\frac{1+\varepsilon_0}{2}}\,M^{\frac{\varepsilon_0}{2}}},
\end{equation*}
for $N$ large enough, where $C>0$ is some universal constant.
\end{proof}
Next, we show that $\sum_{k=1}^M\frac{U_k^2}{2}$ converges to $\frac{\lambda ^2}{2}$ in probability.

\begin{lemma}\label{lemma:trunc}
For every $\delta>0$,
\begin{equation}\label{eq:trunc}
\limsup_{M\to\infty}\limsup_{N\to\infty}\IP\left( \left| \sum_{k=1}^M \frac{U_k^2}{2} - \frac{\lambda^2}{2}\right| >\delta\right)
=0.
\end{equation}
\end{lemma}

\begin{proof}
Take $\alpha>0$ and $\delta>0$.
First, we truncate the variables $U_k$ as follows:
\begin{equation}\label{eq:step3}
  \begin{split}
  \IE\left[ \left| \sum_{k=1}^MU_k^2 - \sum_{k=1}^MU_k^2 \mathds{1}_{U_k^2\leq\frac{\alpha}{M}} \right| \right] &= \IE\left[ \sum_{k=1}^M U_k^2 \mathds{1}_{U_k^2 > \frac{\alpha}{M}} \right] \leq \IE\left[ \sum_{k=1}^M\left| U_k \right|^{2+\varepsilon_0}\frac{M^{\frac{\varepsilon_0}{2}}}{\alpha ^{\frac{\varepsilon_0}{2}}} \right]\overset{\eqref{eq:fourth_moment}}{\leq} \frac{c_{\hat\beta}}{\alpha^{\frac{\varepsilon_0}{2}}},
  \end{split}
\end{equation}
This also implies that
\begin{equation}\label{eq:est_mean}
\left|\IE\left[ \sum_{k=1}^MU_k^2\right]- \IE\left[ \sum_{k=1}^MU_k^2 \mathds{1}_{U_k^2\leq\frac{\alpha}{M}} \right]\right|\leq \IE\left[ \left| \sum_{k=1}^MU_k^2 - \sum_{k=1}^MU_k^2 \mathds{1}_{U_k^2\leq\frac{\alpha}{M}} \right| \right] \overset{\eqref{eq:step3}}{\leq} \frac{c_{\hat\beta}}{\alpha^{\frac{\varepsilon_0}{2}}}.
\end{equation} 
Convergence in \eqref{eq:variance}, together with \eqref{eq:est_mean}, gives
\begin{equation}\label{eq:mean_conv_trunc}
  \limsup_{M\to\infty}\limsup_{N\to\infty}\left|\IE\left[ \sum_{k=1}^MU_k^2\mathds{1}_{U_k^2\leq\frac{\alpha}{M}}\right] - \lambda^2\right| \leq \frac{c_{\hat \beta}}{\alpha^{\frac{\varepsilon_0}{2}}}.
\end{equation}
We now conclude by proving that $ \sum_{k=1}^MU_k^2 \mathds{1}_{U_k^2\leq\frac{\alpha}{M}}$ concentrates around its mean. As the random variables $U_k^2\mathds{1}_{U_k^2\leq\frac{\alpha}{M}}-\IE\left[ U_k^2\mathds{1}_{U_k^2\leq\frac{\alpha}{M}}\right]$ are independent and centered, we have by Chebyshev's inequality:
\begin{align}
&\IP\left( \left| \sum_{k=1}^M U_k^2\mathds{1}_{U_k^2\leq\frac{\alpha}{M}}-\IE\left[ \sum_{k=1}^M U_k^2\mathds{1}_{U_k^2\leq\frac{\alpha}{M}}\right] \right|>\delta \right)\nonumber\\
&{\leq} \delta^{-2}\sum_{k=1}^M\IE \left[ \left( U_k^2\mathds{1}_{U_k^2\leq\frac{\alpha}{M}}-\IE\left[ U_k^2\mathds{1}_{U_k^2\leq\frac{\alpha}{M}}\right]\right)^2 \right]\nonumber \\
&\leq\delta^{-2}\sum_{k=1}^M \left(\IE \left[ \left( U_k^2\mathds{1}_{U_k^2\leq\frac{\alpha}{M}}\right)^2 \right] + \IE \left[  U_k^2\mathds{1}_{U_k^2\leq\frac{\alpha}{M}}\right]^2\right) \leq 2\frac{\alpha^2}{\delta^2 M}.\label{eq:prob}
\end{align}
Putting \eqref{eq:step3}, \eqref{eq:mean_conv_trunc} and \eqref{eq:prob} together, we obtain that
\begin{gather*}
  \limsup_{M\to\infty}\limsup_{N\to\infty} \IP\left( \left| \sum_{k=1}^M \frac{U_k^2}{2} - \frac{\lambda^2}{2}\right| >\delta\right) 
\leq \frac{2c_{\hat\beta}}{\delta \alpha^{\frac{\varepsilon_0}{2}}}.
\end{gather*}
As $\alpha>0$ is arbitrary, this concludes the proof.
\end{proof}

\section{Proof of Theorem \ref{th:main-th}-\eqref{eq:CLT0}}

The proof follows from Proposition \ref{prop:step2}, Lemma \ref{lemma:wass} (along with \eqref{eq:duality}) and Lemma \ref{lemma:trunc}.

\appendix
\section{The Wasserstein distance}
\label{appendix}
\setcounter{equation}{0}
\numberwithin{equation}{section}

Recall the definition of Wasserstein distance \cite{villani}, which is used in Lemma \ref{lemma:wass}.

\begin{definition*}[Wasserstein distances]
Let $(\mathcal{X},d)$ be a Polish metric space and let $p\in [1,\infty)$. For any two probability measures $\mu$ and $\nu$ on $\mathcal{X}$, the Wasserstein distance of order $p$ between $\mu$ and $\nu$ is defined by the formula
\begin{equation}\label{eq:wass_def}
\mathcal{W}_p (\mu,\nu):= \inf \left\{ \IE\left[ d(X,Y)^p \right]^{\frac{1}{p}}, \quad \mathcal{L}(X)=\mu \text{ and } \mathcal{L}(Y)=\nu \right\},
\end{equation}
where by $\IE[\cdot]$ we are referring to an expectation with respect to a coupling of $X$ and $Y$.
\end{definition*}
\begin{remark*}
We will work with $p=1$ and with $d(x,y)=|x-y|$, so that
\begin{equation}\label{eq:wass1_def}
\mathcal{W}_1(\mu,\nu) := \inf \left\{ \IE\left[ |X-Y| \right], \quad \mathcal{L}(X)=\mu \text{ and } \mathcal{L}(Y)=\nu \right\}.
\end{equation}
\end{remark*}

The following property, which is immediate, shows that weak convergence is implied by convergence in the $L^1$-Wasserstein distance:
\begin{equation}\label{eq:duality}
 \sup \left\{ \IE_{\mu}\left[ f(X) \right]-\IE_{\nu}\left[ f(Y) \right], \text{ } f \text{ is bounded and } \|f\|_{\text{Lip}}\leq 1 \right\} \leq \mathcal{W}_1(\mu,\nu).
\end{equation}

  \bibliographystyle{plain}
{\footnotesize \bibliography{polymeres-bib}
}

\end{document}